\documentclass[11pt]{amsart}
\usepackage{amsmath}
\usepackage{amssymb}
\usepackage{graphicx}
\usepackage[colorlinks=true,linkcolor=blue,citecolor=blue]{hyperref}
\usepackage[mathscr]{euscript}
\usepackage{lipsum}
\usepackage{amsrefs}
\usepackage{mathrsfs}
\usepackage{multirow}
\newtheorem{theorem}{Theorem}[section]

\newtheorem{lemma}[theorem]{Lemma}

\theoremstyle{definition}
\newtheorem{definition}[theorem]{Definition}
\newtheorem{example}[theorem]{Example}
\newtheorem{remark}[theorem]{Remark}

\numberwithin{equation}{section}
\title[A Neural Network]{A Neural Network for Solving Inverse Quasi-Variational Inequalities}

\author[Soumitra Dey]{Soumitra Dey$^*$}\thanks{$^*$Corresponding author.}
\address[S. Dey]{Department of Mathematics, The Technion -- Israel Institute of Technology, 3200003 Haifa, Israel}
\email{\tt deysoumitra2012@gmail.com}

\author{Simeon Reich}
\address[S. Reich]{Department of Mathematics, The Technion -- Israel Institute of Technology, 3200003 Haifa, Israel}
\email{\tt sreich@technion.ac.il}



\begin{document}

\begin{abstract}
We study the existence and uniqueness of solutions to the inverse quasi-variational inequality problem. Motivated by the neural network approach to solving optimization problems such as variational inequality, monotone inclusion, and inverse variational problems, we consider a neural network associated with the inverse quasi-variational inequality problem, and establish the existence and uniqueness of a solution to the proposed network. We prove that every trajectory of the proposed neural network converges to the unique solution of the inverse quasi-variational inequality problem and that the network is globally asymptotically stable at its equilibrium point. We also prove that if the function which governs the inverse quasi-variational inequality problem is strongly monotone and Lipschitz continuous, then the network is globally exponentially stable at its equilibrium point. We discretize the network and show that the sequence generated by the discretization of the network converges strongly to a solution of the inverse quasi-variational inequality problem under certain assumptions on the parameters involved. Finally, we provide numerical examples to support and illustrate our theoretical results.
\end{abstract}

\maketitle

%====================================================================================================

\section{Introduction}
\label{Sec:1}

Let $\mathbb{R}^n$ denote the real $n$-dimensional Euclidean space with the usual inner product and norm $\langle\cdot,\cdot\rangle$ and $\|\cdot\|$, respectively. That is, given
$x=(x_1,\cdots,x_n)^\top$ and $y=(y_1,\cdots,y_n)^\top$ in $\mathbb{R}^n$ (here $^\top$ stands for the transpose), we have
$$\langle x,y\rangle=\sum_{i=1}^n x_iy_i \quad {\rm and}\quad
\|x\|=\left(\sum_{i=1}^n x_i^2\right)^{\frac12}.$$

Inverse variational inequalities were introduced by He et al. \cite{BHE2010}.
Recall that an {\em inverse variational inequality} (IVI) calls for finding a point $x^*\in \mathbb{R}^n$ such that
\begin{equation}\label{IVIP}
f(x^*)\in {\Omega} \quad{\rm and}\quad \left\langle x^*,y-f(x^*) \right\rangle\geq 0\quad\forall y\in {\Omega},
\end{equation}
where $f$ is a mapping from $\mathbb{R}^n$ into itself,
and $\Omega$ is a nonempty, closed and convex subset of $\mathbb{R}^n$.
IVIs have found application in various fields, such as traffic network problems \cite{DAUS2013} and economic equilibrium problems \cite{YHAN2017}.

If an inverse function $x=f^{-1}(u)=F(u)$ exists, then
the above IVI problem can be transformed into the following regular variational inequality: find a point $u^*\in \Omega$ such that
\begin{align}
\left\langle F(u^*), v-u^*\right\rangle\geq 0 \; \quad\forall v\in\Omega.
\end{align}
However, in real life, it is common that even if it exists, the inverse function $f^{-1}$ does not have an explicit form. For instance, in network equilibrium problems, the state variables $u$ represent an equilibrium flow pattern under the given controls $x$. It is almost impossible to find an explicit functional relationship between $x$ and $u$ (see \cite{XHE2011}). From the above discussion, it follows that the inverse variational inequality problem is in some sense a particular case of the variational inequality problem (VIP). The VIP is an important tool for solving many optimization problems such as systems of linear equations, systems of nonlinear equations and complementarity problems (see, for instance, \cite{DKIN1980}). After the variational inequality problem was introduced by Stampacchia \cite{GSTA1964}, many researchers have focused their attention on it. It has also been generalized in various directions (see, for example, \cite{YCEN2010, QLDO2017, FFAC2003}) in both finite and infinite dimensional settings. Moreover, many  analytical as well as numerical methods have been introduced for solving the VIP
(see, for instance, \cite{MDNO2003, YCEN2012, HYAM2001, YCEN2011, YCEN22011, BSHE1997, QLDO2019 }).

One of the most important generalizations of the variational inequality problem is the quasi-variational inequality problem.
Let $f:\mathbb{R}^n\rightarrow\mathbb{R}^n$ be a single-valued mapping and let $\Phi:\mathbb{R}^n\rightarrow2^{\mathbb{R}^n}$ be a set-valued mapping.
Then the {\em quasi-variational inequality problem} (QVIP) consists of finding a point $x^*\in\Phi(x^*)$ such that
\begin{align}\label{QVIP:P}
\left\langle f(x^*), y-x^*\right\rangle\geq 0 \; \quad\forall y\in\Phi(x^*).
\end{align}

It is clear that if $\Phi(x)=\Omega$ for every $x\in\mathbb{R}^n$, then the above QVIP (\ref{QVIP:P}) reduces to the regular variational inequality problem \cite{SDEY2019}.  There are several analytical as well as numerical approaches that have been applied by many researchers for solving the above QVIP (see, for instance, \cite{NMIJ2019, YNES2011}  and references therein).

% Out of them, neural network method is an powerful tool to solve QVIP numerically due to its low complexity and faster rate of convergence. In 2018,  \cite{NMIJ2019} has established a neural network model to solve QVIP (\ref{QVIP:P}). Afterword, Nguyen et al. \cite{LVNG2019} introduced a neural network for QVIP (\ref{QVIP:P}) and established the global asymptotic stability and exponential stability at an equilibrium problem of the proposed network under certain assumptions on the mappings involved.

It is not difficult to see \cite{XZOU2016} that IVI (\ref{IVIP}) is equivalent to the following projected equation:
\begin{align}
f(x^*)=P_{\Omega}(f(x^*)-\alpha x^*),
\end{align}
where $\alpha>0$ is any fixed constant and $P_{\Omega}: \mathbb{R}^n \rightarrow \Omega$ is the nearest point projection from $\mathbb{R}^n$
onto $\Omega$ defined by
$$P_{\Omega}(x) := \arg\min_{y\in \Omega} \|x-y\| \; \quad x\in \mathbb{R}^n.$$

Iterative methods for solving IVIs can be found, for example, in \cite{XHE2011, XHU2012}
and references therein. But so far, there is a limited number of numerical algorithms for solving IVIs.

In 2019 Zou et al. \cite{XZOU2016} proposed a novel neural network method for solving IVI (\ref{IVIP}) by
considering the following neural network:
\begin{equation}\label{NN}
\dot{x}=\lambda\left\lbrace P_{\Omega}(f(x)-\alpha x)-f(x)\right\rbrace=S(x),
\end{equation}
where $\dot{x}=\frac{dx}{dt}$ and $\lambda>0$ is a fixed parameter.

It is evident that $x^*$ is a solution of IVI (\ref{IVIP}) if and only if it is an equilibrium point
of the neural network (\ref{NN}), that is, the constant curve $x(t)\equiv x^*$ is a trajectory
of the dynamical system (\ref{NN}); in other words, $S(x^*)=0$.

However, there is a limited number of analytical as well as numerical methods for solving IVIs. Still many researchers have paid attention to them because of their applications in various fields. They also generalized them in various ways (see, for instance, \cite{DAUS2013, YHAN2017, RHU2016}). One of the important generalizations of IVIs is known as the inverse quasi-variational inequality problem (see \cite{DAUS2013, YHAN2017}). An {\em inverse quasi-variational inequality problem} $\text{(IQVIP)}$ consists of finding a point $x^*\in\mathbb{R}^n$ such that
\begin{align}\label{IQVIP:P}
f(x^*)\in \Phi(x^*)\text{ and }\left\langle x^*,y-f(x^*)\right\rangle\geq 0\quad\forall y\in \Phi(x^*),
\end{align}
where $\Phi:\mathbb{R}^n\rightarrow 2^{\mathbb{R}^n}$ is a set-valued map and $f:\mathbb{R}^n\rightarrow \mathbb{R}^n$ is a single-valued mapping.

If $\Phi(x)=\Omega$ for all $x\in\mathbb{R}^n$, then the  IQVIP (\ref{IQVIP:P}) reduces to the IVI (\ref{IVIP}).  That is, (\ref{IQVIP:P}) reduces to the problems of finding a point $x^*\in\mathbb{R}^n$ such that
\begin{align}
f(x^*)\in\Omega\quad\text{and}\quad\left\langle x^*, y-f(x^*)\right\rangle\geq 0 \; \quad\forall y\in\Omega.
\end{align}

Aussel et al. \cite{DAUS2013} studied the IQVIP with an application to road pricing problems and Han et al. \cite{YHAN2017} established the existence of a solution to the inverse quasi-variational inequality problem by using a fixed point theorem and the Fan-Knaster-Kuratowski-Mazurkiewicz (KKM) Lemma.  Subsequently, the existence result of Han et al. \cite{YHAN2017} was extended by Dey et al. \cite{SDEY2018}. More recently, the inverse quasi-variational inequality problem has been revisited and generalized by many researchers, and vector-valued IQVIPs have also been introduced (see \cite{SSCH2021, ZBWA2019}) and references therein.

In the case where the set-valued mapping $\Phi:\mathbb{R}^n\rightarrow 2^{\mathbb{R}^n}$ which governs the IQVIP (\ref{IQVIP:P}) has nonempty, closed and convex point values, it is not difficult to check that $x^*$ is a solution to $(\ref{IQVIP:P})$ if and only if it is a solution to the projection equation
\begin{align}
f(x)=P_{\Phi(x)}(f(x)-\alpha x),
\end{align}
where $\alpha>0$ is a fixed constant.

Recently, neural networks have found very effective applications in signal processing, pattern recognition, associative memory, as well as in other engineering or scientific fields \cite{HGZH2010, QSLI2013, HZHA2011}. Neural network methods have also been applied to solving mathematical programming and related optimization problems such as variational inequalities \cite{XBGA2005}, quasi-variational inequalities \cite{LVNG2019, NMIJ2019}, and variational inclusions \cite{SDEY2019}.
Motivated by this state of affairs, in the present paper we propose the following neural network for solving the IQVIP $(\ref{IQVIP:P})$:
\begin{equation}\label{IQVIP}
\dot{x}=\lambda(t)\left\lbrace P_{\Phi(x)}(f(x)-\alpha x)-f(x)\right\rbrace=S(x, t),
\end{equation}
where $\dot{x}=\frac{dx}{dt}$ and $\lambda(t)>0$, $t\geq 0$, are parameters.

In particular, if $\lambda(t)=\lambda$ for every $t\in [0, \infty)$ and $\Phi(x)=\Omega$ for every $x\in\mathbb{R}^n$, then the network $(\ref{IQVIP})$ reduces to the network $(\ref{NN})$, which was considered and studied in (\cite{XZOU2016, HKXU2020}).

Our paper is organized as follows. In section \ref{Sec:2}, we recall some basic definitions and results for further use. In section \ref{Sec:3}, we study the existence and uniqueness of a solution to IQVIP $(\ref{IQVIP:P})$. In section \ref{Sec:4},  we provide the existence and uniqueness of a solution to the network $(\ref{IQVIP})$. In Section \ref{Sec:5}, we study the global stability of the neural network $(\ref{IQVIP})$ including global asymptotic stability and global exponential stability. In section \ref{Sec:6}, we discretize the network (\ref{IQVIP}) and establish the convergence of the sequence generated by its discretization. In section \ref{Sec:7}, we provide a numerical example to illustrate the effectiveness of the neural network in solving IQVIP $(\ref{IQVIP:P})$. Finally, in section \ref{Sec:8}, we draw a conclusion from our analysis.

%==================================================================================================================

\section{Preliminaries}
\label{Sec:2}

In this section we collect some basic definitions and results which are used in the sequel.
For more details one can refer to \cite{FFAC2003, HHBA2011, NMIJ2019, PBEE1975, XZOU2016, MANR1985}.

\begin{definition}
A function $f:\mathbb{R}^n\rightarrow\mathbb{R}^n$ is said to be {\em Lipschitz} continuous with constant $L$ on $\mathbb{R}^n$
if, for every pair of points $x, y\in\mathbb{R}^n$, we have
\begin{align*}
\|f(x)-f(y)\|\leq L\|x-y\|.
\end{align*}
\end{definition}

\begin{definition}
A function $f:\mathbb{R}^n\rightarrow\mathbb{R}^n$ is said to be {\em nonexpansive} if it is Lipschitz continuous with constant equal to one,
that is,
$$\|f(x)-f(y)\|\le \|x-y\| \; \quad \forall x,y\in \mathbb{R}^n.$$
\end{definition}

\begin{definition}
A function $f: \mathbb{R}^n\rightarrow\mathbb{R}^n$ is said to be monotone
if
\begin{align*}
\langle f(x)-f(y), x-y\rangle\geq 0 \; \quad  \forall x,y\in \mathbb{R}^n.
\end{align*}
\end{definition}

\begin{definition}
A function $f: \mathbb{R}^n\rightarrow\mathbb{R}^n$ is said to be {\em $\beta$-strongly monotone}
if, for some $\beta>0$, we have
\begin{align*}
\langle f(x)-f(y), x-y\rangle\geq\beta\|x-y\|^2 \quad \forall x,y\in \mathbb{R}^n.
\end{align*}
\end{definition}

\noindent Note that strongly monotone mappings are monotone, but the converse is not true.
\begin{remark}\label{rem:1}
Note that if a function $f: \mathbb{R}^n\rightarrow\mathbb{R}^n$ is $L$-Lipschitz and $\beta$-strongly monotone,
then we must have $L\ge \beta$.  Indeed,
$$\beta\|x-y\|^2\le\langle x-y, f(x)-f(y)\rangle\le\|x-y\|\cdot\|f(x)-f(y)\|\le L\|x-y\|^2.$$
\end{remark}

The following characterization of the nearest point projection is very well known.

\begin{lemma}\label{plem}
Let $\Omega$ be a nonempty, closed and convex subset of $\mathbb{R}^n$. Given $x\in \mathbb{R}^n$ and $z\in \Omega$, we have
\begin{align}
z=P_{\Omega}(x)\quad\Longleftrightarrow \quad \left\langle x-z,y-z\right\rangle\leq 0 \; \quad \forall y\in \Omega.
\end{align}
It follows that the projection operator $P_{\Omega}$ is nonexpansive.
\end{lemma}

\begin{lemma}
Let $\Phi:\mathbb{R}^n\rightarrow 2^{\mathbb{R}^n}$ be a set-valued mapping with $\Phi(x)=f(x)+\Omega$ for every $x\in\mathbb{R}^n$, where $f:\mathbb{R}^n\rightarrow\mathbb{R}^n$ is a single-valued mapping and $\Omega$ is a nonempty closed convex subset of $\mathbb{R}^n$. Then
\begin{align}\label{plem2}
P_{f(x)+\Omega}(z)=f(x)+P_\Omega(z-f(x))\quad\forall x,z\in\mathbb{R}^n.
\end{align}
\end{lemma}

\begin{definition}
Let $(X, \|\cdot\|)$ be a normed linear space. A mapping $f:X\rightarrow X$ is said to be a strict contraction if for some $\alpha\in [0, 1)$, we have
\begin{align*}
\|f(x)-f(y)\|\leq\alpha\|x-y\| \; \quad\forall x, y\in X.
\end{align*}
\end{definition}

For any given map $f:X\rightarrow X$, the set of all the fixed points of $f$ is denoted by $\text{Fix}(f)$. That is, $\text{Fix}(f) := \left\lbrace x\in X: f(x)=x\right\rbrace.$

\begin{theorem}$\text{(Banach's fixed point theorem)}$\label{Banach:Principle}
Let $X$ be a Banach space and let $f:X\rightarrow X$ be a strict contraction. Then $f$ has a unique fixed point.
\end{theorem}

Recall that, roughly speaking, a dynamical system is a system in which a function describes the time dependence of a point in a given space. At any given time, a dynamical system has a state given by a tuple of real numbers (a vector) that can be represented by a point in an appropriate state space. The study of dynamical systems is the focus of dynamical system theory, which has applications to a wide variety of fields such as mathematics, physics, biology, chemistry, engineering, economics and medicine.

More precisely, a dynamical system is a manifold $S$ called the phase (or state) space endowed with a family of smooth evolution functions $m^t$ that for any element $t\in T$, the time, map a point of the phase space back into the phase space. The notion of smoothness changes with applications and the type of the manifold. There are several choices for the set $T$. When $T$ is taken to be the reals, then the dynamical system is called a flow; and if T is restricted to the non-negative reals, then the dynamical system is said to be a semi-flow. When $T$ is taken to be the integers, it is called a cascade; and the restriction to the non-negative integers is said to be a semi-cascade.

The evolution function $m^t$ is often the solution of a differential equation of motion:
\begin{align}\label{ds}
\dot{x}=h(x).
\end{align}
This equation involves the time derivative, represented by the dot, of a trajectory $x(t)$ on the phase space starting at some point $x_0$. The vector field $h(x)$ is a smooth function that at every point of the phase space $S$ provides the velocity vector of the dynamical system at that point. These vectors are not vectors in the phase space $S$, but in the tangent space $T_xS$ of the point $x$. Given a smooth $m^t$, an autonomous vector field can be derived from it.

The differential equations determining the evolution function $m^t$ are often ordinary differential equations; in this case the phase space $S$ is a finite dimensional manifold. Many of the concepts in dynamical systems can be extended to certain infinite dimensional manifolds -- those that are locally Banach spaces -- in which case the differential equations are partial differential equations. In the late $20th$ century the dynamical system perspective of partial differential equations started gaining popularity.

The concept of a dynamical system has its origins in Newtonian mechanics. There, as in other natural sciences and engineering disciplines, the evolution rule of dynamical systems is an implicit relation that gives the state of the system for only a short time into the future. To determine the state for all future times requires iterating the relation many times -- each iteration advancing time a small step. We refer to such an iterative procedure as solving the system or integrating it. If the system can be solved, then given an initial point, it is possible to determine all its future positions. This collection of points is known as a trajectory or an orbit

%
%Consider a dynamical system
%\begin{equation}\label{ds}
%\dot{x}=h(x),
%\end{equation}
%where $h:\mathbb{R}^n\to \mathbb{R}^n$ is a function.

Recall that a point $x_e\in\mathbb{R}^n$
is said to be an equilibrium point of the dynamical system (\ref{ds}) if $h(x_e)=0$.
In other words, the constant curve $x(t)\equiv x_e$ is a trajectory.

\begin{definition}
Let $x_e\in\mathbb{R}^n$ be an equilibrium point of the dynamical system (\ref{ds}).
\begin{itemize}
\item
The dynamical system (\ref{ds}) is said to be
globally asymptotically stable at $x_e$ if
every trajectory $x=x(t)$ of (\ref{ds}) converges to $x_e$ as $t\to\infty$.
(This implies that $x_e$ is the unique equilibrium point).

\item
The dynamical system (\ref{ds}) is said to be globally exponentially stable at $x_e$ if every trajectory $x=x(t)$ starting
at any initial point $x(t_0)\in \mathbb{R}^n$ with $t_0\ge 0$ satisfies the estimate
\begin{align*}
\|x(t)-x_e\|\le \eta e^{-\xi(t-t_0)}\quad\forall t\ge t_0,
\end{align*}
where $\eta>0$ and $\xi>0$ are positive constants.
\end{itemize}
\end{definition}

%\begin{lemma}\cite{YNES2011}
%Let $s:\mathbb{R}^n\rightarrow\mathbb{R}^n$ be a Lipschitz continuous function, i. e.,
%\begin{align}
%\|s(x)-s(y)\|\leq\lambda\|x-y\|,\quad\forall x, y\in\mathbb{R}^n.
%\end{align}
%Let $\Omega$ be a nonempty closed convex subset of $\mathbb{R}^n.$ Then
%\begin{align}
%\Phi(x)=s(x)+\Omega
%\end{align}
%satisfies $(\ref{Lem:in1})$ with the same value of $\lambda$.
%\end{lemma}

%=====================================================================================================
\section{Existence and uniqueness of a solution to the IQVIP}
\label{Sec:3}

In this section we establish the existence and uniqueness of the solution to the inverse quasi-variational inequality problem (IQVIP).
We begin with the following lemma.

\begin{lemma}\label{Result:1}
Let $\Phi:\mathbb{R}^n\rightarrow 2^{\mathbb{R}^n}$ be a  set-valued mapping with nonempty, closed and convex point values. Then we have
\begin{align}\label{Proj:lem1}
& \|u-P_{\Phi(x)}(u)-(v-P_{\Phi(y)}(v))\|\\&\leq\|u-v\|+\|P_{\Phi(x)}(v)-P_{\Phi(y)}(v)\| \; \quad\forall x, y, u, v\in\mathbb{R}^n\nonumber.
\end{align}

\begin{proof}
Let $x, y, u, v\in\mathbb{R}^n$. Using Lemma \ref{plem}, we see that
\begin{align}\label{Lem1:in1}
\left\langle v-P_{\Phi(x)}(v), P_{\Phi(x)}(u)-P_{\Phi(x)}(v)\right\rangle\leq 0
\end{align}
and
\begin{align}\label{Lem1:in2}
\left\langle u-P_{\Phi(x)}(u), P_{\Phi(x)}(v)-P_{\Phi(x)}(u)\right\rangle\leq 0.
\end{align}

Adding $(\ref{Lem1:in1})$ and $(\ref{Lem1:in2})$, we get
\begin{align}\label{Lem1:in3}
\|P_{\Phi(x)}(u)-P_{\Phi(x)}(v)\|^2\leq\left\langle P_{\Phi(x)}(u)-P_{\Phi(x)}(v),u-v\right\rangle.
\end{align}

It follows from $(\ref{Lem1:in3})$ that
\begin{align}\label{Lem1:in4}
&\|u-P_{\Phi(x)}(u)-(v-P_{\Phi(x)}(v))\|^2\\&  \leq\|u-v\|^2+\|P_{\Phi(x)}(u)-P_{\Phi(x)}(v)\|^2-2\left\langle P_{\Phi(x)}(u)-P_{\Phi(x)}(v),u-v\right\rangle\nonumber
\\& \leq\|u-v\|^2-\|P_{\Phi(x)}(u)-P_{\Phi(x)}(v)\|^2\nonumber\\&
\leq \|u-v\|^2\nonumber
\end{align}

and so,

\begin{align}\label{Lem:in5}
& \|u-P_{\Phi(x)}(u)-(v-P_{\Phi(y)}(v))\| \\& \leq\|u-P_{\Phi(x)}(u)-(v-P_{\Phi(x)}(v))\|+\|P_{\Phi(x)}(v)-P_{\Phi(y)}(v)\|\nonumber\\&
\leq \|u-v\| +\|P_{\Phi(x)}(v)-P_{\Phi(y)}(v)\|.\nonumber
\end{align}

\end{proof}
\end{lemma}

\begin{theorem}\label{Thm:IQVIP}
Let $\Phi:\mathbb{R}^n\rightarrow2^{\mathbb{R}^n}$ be a set-valued mapping with nonempty, closed and convex point values. Assume that $f:\mathbb{R}^n\rightarrow\mathbb{R}^n$ is $L-$Lipschitz continuous and $\beta-$strongly monotone. Assume further that there exists some $\kappa>0$ such that
\begin{align}\label{Lem:in1}
\|P_{\Phi(x)}(z)-P_{\Phi(y)}(z)\| \leq\kappa\|x-y\| \; \quad\forall x, y, z\in\mathbb{R}^n
\end{align}
and $L^2-2\alpha(\beta-\kappa)<\kappa^2$, where $\alpha>0$ is a constant.
Then the inverse quasi-variational inequality problem $(\ref{IQVIP:P})$ has a unique solution.
\end{theorem}

\begin{proof}
Define a mapping $h:\mathbb{R}^n\rightarrow\mathbb{R}^n$ by
\begin{align*}
h(x)=x-\frac{1}{\alpha}f(x)+\frac{1}{\alpha}P_{\Phi(x)}(f(x)-\alpha x),
\end{align*}
where $\alpha>0$ is a fixed constant.

It is clear that $x^*$ is a solution to the inverse quasi-variational inequality if and only if $x^*$ is a fixed point of the mapping $h$.

Let $\bar{x}=f(x)-\alpha x$ and $\bar{y}=f(y)-\alpha y$. Using Lemma \ref{Result:1}, we see that
{\small
\begin{align}\label{in1}
&\|h(x)-h(y)\|\\& =\|x-\frac{1}{\alpha}f(x)+\frac{1}{\alpha}P_{\Phi(x)}(f(x)-\alpha x)-(y-\frac{1}{\alpha}f(y)+\frac{1}{\alpha}P_{\Phi(y)}(f(y)-\alpha y))\|\nonumber\\&
=\frac{1}{\alpha}\|\bar{y}-\bar{x}+P_{\Phi(x)}(\bar{x})-P_{\Phi(y)}(\bar{y})\|\nonumber\\&
=\frac{1}{\alpha}\|\bar{x}-P_{\Phi(x)}(\bar{x})-(\bar{y}-P_{\Phi(y)}(\bar{y}))\|\nonumber\\&
\leq\frac{1}{\alpha}(\|\bar{x}-\bar{y}\|+\|P_{\Phi(x)}(\bar{y})-P_{\Phi(y)}(\bar{y})\|)\nonumber\\&
\leq\frac{1}{\alpha}(\|\bar{x}-\bar{y}\|+\kappa\|x-y\|).\nonumber
\end{align}}

Now,
\begin{align}\label{in2}
\|\bar{x}-\bar{y}\|^2& =\|f(x)-\alpha x-(f(y)-\alpha y)\|^2\\&
=\|f(x)-f(y)\|^2-2\alpha\left\langle f(x)-f(y), x-y\right\rangle+\alpha^2\|x-y\|^2\nonumber\\&
\leq (L^2-2\beta\alpha+\alpha^2)\|x-y\|^2. \nonumber
\end{align}

Using $(\ref{in1})$ and $(\ref{in2})$ we get,
\begin{align}
\|h(x)-h(y)\|\leq \frac{1}{\alpha}(\sqrt{(L^2-2\beta\alpha+\alpha^2)}+\kappa)\|x-y\|.
\end{align}

It clearly follows from our assumptions that $h$ is a strict contraction with constant $(\sqrt{(L^2-2\beta\alpha+\alpha^2)}+\kappa)/\alpha\in[0,1)$. Therefore, by the Banach contraction principle (Theorem \ref{Banach:Principle}), the mapping $h$ has a unique fixed point. In other words, the inverse quasi-variational inequality problem (\ref{IQVIP:P}) has a unique solution.
\end{proof}

\begin{remark}
Note that assumption $(\ref{Lem:in1})$ is a kind of contraction property for the set-valued mapping $\Phi$ on $\mathbb{R}^n$.
In several applications the point image can be written as
\begin{align*}
\Phi(x)=s(x)+\Omega,
\end{align*}
where $s(x)$ is a Lipschitz continuous single-valued mapping from $\mathbb{R}^n$ into itself with Lipschitz constant $\lambda$ and $\Omega$ is a closed convex subset of $\mathbb{R}^n$ . In this case, the assumption $(\ref{Lem:in1})$ holds with the same Lipschitz constant value of $\lambda.$
\end{remark}

The following example illustrates our Theorem \ref{Thm:IQVIP}.

\begin{example}

Define $s:\mathbb{R}\rightarrow\mathbb{R}$ by $s(x) := \frac{1}{1+|x|}$ and $f:\mathbb{R}\rightarrow\mathbb{R}$
by $f(x) := 2x$, and let $\alpha=2.$

Define $\Phi:\mathbb{R}\rightarrow2^\mathbb{R}$ by $\Phi(x) := s(x)+K$,  where $K=[-1, 1]$.
It is clear that for each $x\in\mathbb{R}$, $\Phi(x)$ is a non-empty, closed and convex subset of $\mathbb{R}$.

We have
\begin{align*}
|s(x)-s(y)|&=\Big\vert\frac{1}{1+|x|}-\frac{1}{1+|y|}\Big\vert\\&\leq \vert \vert x\vert -\vert y\vert\vert\leq \vert x-y\vert.
\end{align*}

Therefore the function $s$ is Lipschitz continuous with constant $\kappa=1$.

It is not dificult to check that $f$ is $2$-Lipschitz continuous, $2$-strongly monotone and that
\begin{align*}
L^2-2\alpha(\beta-\kappa)<\kappa^2.
\end{align*}

Therefore, by Theorem \ref{Thm:IQVIP}, IQVIP (\ref{IQVIP:P}) has a unique solution.

In particular, $x^*=0$ is the unique solution of the IQVIP (\ref{IQVIP:P}).
\end{example}

%===================================================================================================
\section{Existence and uniqueness}
\label{Sec:4}

In this section we establish the existence and uniqueness of the solution to the proposed neural network $(\ref{IQVIP})$.

\begin{theorem}\label{ExT}
Let $\Phi:\mathbb{R}^n\rightarrow 2^{\mathbb{R}^n}$ be a set-valued mapping with nonempty, closed and convex point images and let $f:\mathbb{R}^n\rightarrow \mathbb{R}^n$ be a Lipschitz continuous mapping with Lipschitz constant $L$. Assume that there exists a number $\kappa>0$ such that
\begin{align}
\|P_{\Phi(x)}(z)-P_{\Phi(y)}(z)\| \leq\kappa\|x-y\| \;\quad\forall x, y, z\in\mathbb{R}^n.
\end{align}
Then the dynamical system $(\ref{IQVIP})$ a has unique solution.
\end{theorem}

\begin{proof}
We claim that $S(x, t)$ is Lipschitz continuous for all fixed $t\geq 0$. Indeed, we have
\begin{align*}
&\|S(x, t)-S(y, t)\|\\&=\|\lambda(t)( P_{\Phi(x)}(f(x)-\alpha x)-f(x))-\lambda(t)(P_{\Phi(y)}(f(y)-\alpha y)-f(y))\|\\&
\leq \lambda(t)\|f(x)-f(y)\|+\lambda(t)\|P_{\Phi(x)}(f(x)-\alpha x)-P_{\Phi(y)}(f(y)-\alpha y)\|\\&
\leq \lambda(t)\|f(x)-f(y)\|+\lambda(t)[\|P_{\Phi(x)}(f(x)-\alpha x)-P_{\Phi(x)}(f(y)-\alpha y)\|\\&+\|P_{\Phi(x)}(f(y)-\alpha y)-P_{\Phi(y)}(f(y)-\alpha y)\|] \\&
\leq \lambda(t)\|f(x)-f(y)\|+\lambda(t)[\|(f(x)-\alpha x)-(f(y)-\alpha y)\|+\kappa\|x-y\|] \\&
\leq(2L+\alpha+\kappa)\lambda(t)\|x-y\|.
\end{align*}

Furthermore, if $\lambda(t)$ is continuous, then the function $S(x,\cdot)$ is continuous for all fixed $x\in\mathbb{R}^n$ and the differential equation $(\ref{IQVIP})$,
for arbitrary initial points $x_0\in\mathbb{R}^n$, has a unique solution for all $t\geq t_0\geq 0$ (see \cite{HART2002}).
\end{proof}

%===========================================================================================
\section{Stability Analysis}
\label{Sec:5}

In this section we discuss the global asymptotic and exponential stability of the neural network $(\ref{IQVIP})$.

It is known that Lyapunov functions play a key role in the stability analysis of dynamical systems.
Recall that a function $V: \mathbb{R}^n\to \mathbb{R}$ is said to be a Lyapunov function (about $x=x_e$) for
the dynamical system (\ref{ds}) if the following three properties are satisfied:
\begin{itemize}
\item[(L1)] $V$ is positive definite, namely, $V(x)\ge 0$ for all $x\in \mathbb{R}^n$
and  $V(x)=0$ if and only if $x=x_e$;
\item[(L2)]
$\dot{V}$ is negative definite along the trajectories of (\ref{ds}), that is, if $x(t)$ is a
trajectory of (\ref{ds}), then $\dot{V}(x(t))\le 0$ for all $t\ge 0$ and $\dot{V}(x)<0$ for all $x\neq x_e$;
\item[(L3)]
$V$ is coercive (also known as radially unbounded), that is, $V(x)\to\infty$ as $\|x\|\to\infty$.
\end{itemize}

The following result is one of the fundamental theorems in the theory dynamical systems.

\begin{theorem} \label{T:Lya} (Lyapunov's Theorem)
Let $x_e$ be an equilibrium of the dynamical system (\ref{ds}).
If there exists a Lyapunov function about $x_e$, then the dynamical system (\ref{ds})
globally asymptotically stable at the equilibrium point $x_e$.
\end{theorem}

We use the above Theorem \ref{T:Lya} to show the stability of the solution to the system $(\ref{IQVIP})$.

\begin{theorem}\label{Result4}
Let $\Phi:\mathbb{R}^n\rightarrow 2^{\mathbb{R}^n}$ be a set-valued map with nonempty, closed and convex point values and let $f:\mathbb{R}^n\rightarrow \mathbb{R}^n$ be  Lipschitz continuous with Lipschitz constant $L$ and $\beta$-strongly monotone. Assume that the parameters $\lambda(t)\in C([0, \infty))$.
Assume that
\begin{align}
&1+2\kappa-2\beta+\alpha^2+L^2-2\alpha\beta<0\nonumber,\\&
\int_{t_0}^\infty \lambda(t)dt=+\infty\nonumber
\end{align}
and
\begin{align}
\hspace{-2cm} L^2-2\alpha(\beta-\kappa)<\kappa^2\nonumber,
\end{align}
where $\kappa$ satisfies
\begin{align*}
\|P_{\Phi(x)}(z)-P_{\Phi(y)}(z)\| \leq\kappa\|x-y\|\quad\forall x, y, z\in\mathbb{R}^n.
\end{align*}
Then the dynamical system $(\ref{IQVIP})$ converges to the solution of IQVIP (\ref{IQVIP:P}) at the rate
\begin{align*}
\|x(t)-x^*\|\leq\|x_0-x^*\|\exp{\int_{t_0}^t \Lambda(t)dt},
\end{align*}
where $\Lambda(t)=\lambda(t)[1+2\kappa-2\beta+\alpha^2+L^2-2\alpha\beta]$.

Furthermore, the dynamical system $(\ref{IQVIP})$ is globally asymptotically at the point of equilibrium $x^*$. In addition, if $\lambda(t)\geq\lambda^*>0$ for every $t\geq 0$, then the dynamical system $(\ref{IQVIP})$ is globally exponentially stable at the point of equilibrium $x^*$.
\end{theorem}

\begin{proof}
Using Theorem \ref{ExT}, we can easily show that $(\ref{IQVIP})$ has a unique solution. Also, under our assumptions, it is not difficult to check that the neural network $(\ref{IQVIP})$ has a unique equilibrium point, that is, the IQVIP (\ref{IQVIP:P}) has a unique solution. Let $x^*$ be the unique solution of the IQVIP (\ref{IQVIP:P}). Now we have to show that the trajectories of the network are globally asymptotically stable at the equilibrium point $x^*$. To this end, consider the Lyapunov function
\begin{align*}
V(x)=\|x-x^*\|^2.
\end{align*}
We have
{\small
\begin{align*}
&\dot{V}(x)=2 \left\langle x-x^*, x^\prime\right\rangle\\&
=2\left\langle x-x^*, \lambda(t)(P_{\Phi(x)}(f(x)-\alpha x)-f(x))\right\rangle\\&
=2\left\langle x-x^*, \lambda(t)(P_{\Phi(x)}(f(x)-\alpha x)-f(x))-\lambda(t)(P_{\Phi(x^*)}(f(x^*)-\alpha x^*)-f(x^*))\right\rangle\\&
=-2\lambda(t)\left\langle  x-x^*, f(x)-f(x^*)\right\rangle\\&+2\lambda(t)\left\langle x-x^*,P_{\Phi(x)}(f(x)-\alpha x)-P_{\Phi(x^*)}(f(x^*)-\alpha x^*)\right\rangle\\&
=-2\lambda(t)\left\langle  x-x^*, f(x)-f(x^*)\right\rangle\\&
+2\lambda(t)\left\langle x-x^*,P_{\Phi(x)}(f(x)-\alpha x)-P_{\Phi(x)}(f(x^*)-\alpha x^*)\right\rangle\\&
+2\lambda(t)\left\langle x-x^*,P_{\Phi(x)}(f(x^*)-\alpha x^*)-P_{\Phi(x^*)}(f(x^*)-\alpha x^*)\right\rangle\\&
\leq -2\lambda(t)\beta\|x-x^*\|^2+\lambda(t) [\|x-x^*\|^2\\&+\|P_{\Phi(x)}(f(x)-\alpha x)-P_{\Phi(x)}(f(x^*)-\alpha x^*)\|^2]+2\lambda(t)\kappa\|x-x^*\|^2\\&
\leq \lambda(t)[1+2\kappa-2\beta+\alpha^2+L^2-2\alpha\beta]\|x-x^*\|^2\\&
=\lambda(t)[1+2\kappa-2\beta+\alpha^2+L^2-2\alpha\beta]V(x)=\Lambda(t)V(x),
\end{align*}}
where $\Lambda(t)=\lambda(t)[1+2\kappa-2\beta+\alpha^2+L^2-2\alpha\beta]$.

Since $\int_{t_0}^\infty \lambda(t)dt=+\infty$ and $1+2\kappa-2\beta+\alpha^2+L^2-2\alpha\beta<0$, we see that $\int_{t_0}^\infty \Lambda(t)dt=-\infty$. Hence $\exp{\int_{t_0}^\infty \Lambda(t)dt}=0.$

Consequently, the trajectory $x(t)$ converges to the unique solution $x^*$ of $(\ref{IQVIP:P})$ and it is not difficult to show that
\begin{align}\label{asymp:inq}
\|x(t)-x^*\|\leq\|x_0-x^*\|\exp{\int_{t_0}^t \Lambda(t)dt}.
\end{align}

It now follows from Theorem \ref{T:Lya} that the dynamical system $(\ref{IQVIP})$ is globally asymptotically stable at the equilibrium point $x^*$.

If $\lambda(t)\geq\lambda^*>0$ for every $t\geq 0$, from (\ref{asymp:inq}) we get
\begin{align*}
\|x(t)-x^*\|\leq\|x_0-x^*\|e^{-\zeta(t-t_0)}\quad\forall t\geq t_0,
\end{align*}
where $\zeta=-\lambda^*(1+2\kappa-2\beta+\alpha^2+L^2-2\alpha\beta)>0$.

Therefore, the dynamical system $(\ref{IQVIP})$ is globally exponentially stable at the equilibrium point $x^*$, as asserted.

\end{proof}

\begin{remark}
In this section we have analyzed the stability of the continuous recurrent neural network (\ref{IQVIP}) for solving the inverse quasi-variational inequality problem (\ref{IQVIP:P}) with the variable step size parameter $\lambda(t)$, $t\in [0, \infty)$. We have shown that if the function which governs the IQVIP (\ref{IQVIP:P}) is strongly monotone and Lipschitz continuous, then the network (\ref{IQVIP}) is globally asymptotically and globally exponentially stable at its equilibrium point.
\end{remark}

%================================================================================================================
\section{Discretization of the Network (\ref{IQVIP})}
\label{Sec:6}
In this section we study the discretization method for solving the neural network (\ref{IQVIP}) and establish its strong convergence under certain assumptions on the parameters involved.

The explicit discretization of the neural network (\ref{IQVIP}) with respect to $t$ with the step-size $h_n$ and with the initial point $x_0\in\mathbb{R}^n$ is as follows:
\begin{align}
\frac{x_{n+1}-x_n}{h_n}=\lambda_n \{P_{\Phi(x_n)}(f(x_n)-\alpha x_n)-f(x_n)\},
\end{align}

or, equivalently,
\begin{align}\label{mainin1}
x_{n+1}=x_n+\lambda_n h_n\{P_{\Phi(x_n)}(f(x_n)-\alpha x_n)-f(x_n)\}.
\end{align}

If $h_n=1$, then the above scheme reduces to the following one:
\begin{align}\label{Sub:mainin1}
x_{n+1}=x_n+\lambda_n \{P_{\Phi(x_n)}(f(x_n)-\alpha x_n)-f(x_n)\}.
\end{align}

In this section we study the above modified gradient projection method with variable step sizes.

\begin{theorem}\label{Dis:Thm}
Assume that
\begin{enumerate}
\item $f:\mathbb{R}^n\rightarrow\mathbb{R}^n$ is $\beta$-strongly monotone and $L$-Lipschitz continuous;
\item $\Phi(x)=s(x)+\Omega$, where $s:\mathbb{R}^n\rightarrow\mathbb{R}^n$ is a Lipschitz continuous mapping with Lipschitz constant $l>0$
and $\Omega$ is a nonempty, closed and convex subset of $\mathbb{R}^n$;
\item
\begin{align}
\beta>l,\quad
\alpha>\frac{L^2+l^2}{2(\beta-l)};\label{Asump:trick1}
\end{align}
\item for every $n\in\mathbb{N}$,
\begin{align}
& 0<A<\lambda_n<B,\label{Asump:trick2}\quad \text{where}\\&
0<\frac{B^2}{A}<\frac{2\alpha(\beta-l)-(L^2+l^2)}{\alpha^2 (\beta-l)}\label{Asump:trick3}.
\end{align}
\end{enumerate}
Then the sequence $\{x_n\}$ generated by (\ref{Sub:mainin1}) converges strongly to the unique solution of the IQVIP (\ref{IQVIP:P}).
\end{theorem}

\begin{proof}

Using our above assumptions, and the existence and uniqueness Theorem \ref{Thm:IQVIP} for the IQVIP (\ref{IQVIP:P}), we conclude that it has a unique solution.
Let $x^*$ be the unique solution of the IQVIP (\ref{IQVIP:P}). Recalling (\ref{Sub:mainin1}), we have
\begin{align*}
x_{n+1}=x_n+\lambda_n \{P_{\Phi(x_n)}(f(x_n)-\alpha x_n)-f(x_n)\}
\end{align*}

and therefore
\begin{align}\label{maininq}
\|x_{n+1}-x^*\|^2&=\|x_n-x^*+\lambda_n\{P_{\Phi(x_n)}(f(x_n)-\alpha x_n)-f(x_n)\}\|^2\\&
=\|x_n-x^*\|^2+\lambda_n^2\|P_{\Phi(x_n)}(f(x_n)-\alpha x_n)-f(x_n)\|^2\nonumber\\&+2\lambda_n\left\langle x_n-x^*, P_{\Phi(x_n)}(f(x_n)-\alpha x_n)-f(x_n)\right\rangle\nonumber.
\end{align}

Setting $y_n=P_{\Phi(x_n)}(f(x_n)-\alpha x_n)$,
we see that we need to approximate $\|y_n-f(x_n)\|^2$.

To this end, we first note that
\begin{align}
y_n&=P_{s(x_n)+\Omega}(f(x_n)-\alpha x_n)\\&
=s(x_n)+P_{\Omega}(f(x_n)-\alpha x_n-s(x_n))\nonumber,
\end{align}
which is equivalent to
\begin{align}\label{inq1}
y_n-s(x_n)=P_{\Omega}(f(x_n)-\alpha x_n-s(x_n)).
\end{align}
Therefore, using the characterization of the nearest point projection, we see that, for any $n\in\mathbb{N}$,
\begin{align}\label{inq2}
\left\langle f(x_n)-\alpha x_n-y_n,z-y_n+s(x_n)\right\rangle\leq 0 \quad\forall z\in\Omega.
\end{align}

Since $x^*$ is a solution of the IQVIP (\ref{IQVIP:P}),
we have $x^*\in\mathbb{R}^n$ and
\begin{align}\label{inq3}
f(x^*)=P_{\Phi(x^*)}(f(x^*)-\alpha x^*).
\end{align}

Solving equation (\ref{inq3}) amounts to solving the following problem:
find $x^*\in\mathbb{R}^n$ such that
\begin{align}\label{inq4}
f(x^*)\in\Phi(x^*)\quad\text{and}\quad\left\langle x^*, f(x^*)-y\right\rangle\leq 0\quad\forall y\in\Phi(x^*).
\end{align}

Since $y_n-s(x_n)\in\Omega$, we have
\begin{align}\label{inq5}
y_n-s(x_n)+s(x^*)\in\Phi(x^*)\quad\forall n\in\mathbb{N}.
\end{align}

Using $(\ref{inq4})$ and $(\ref{inq5})$, we get
\begin{align}\label{inq6}
\left\langle \alpha x^*, f(x^*)-y_n+s(x_n)-s(x^*)\right\rangle\leq 0\quad\forall n\in\mathbb{N}.
\end{align}

Again, since $f(x^*)\in\Phi(x^*)$, we have $f(x^*)-s(x^*)\in\Omega$. Therefore it follows from (\ref{inq2}) that
\begin{align}\label{inq7}
\left\langle f(x_n)-\alpha x_n-y_n,f(x^*)-y_n+s(x_n)-s(x^*)\right\rangle\leq 0\quad\forall n\in\mathbb{N}.
\end{align}

Combining (\ref{inq6}) and (\ref{inq7}), we have
\begin{align}\label{inq8}
\left\langle f(x_n)-\alpha x_n-y_n+\alpha x^*,f(x^*)-y_n+s(x_n)-s(x^*)\right\rangle\leq 0\quad\forall n\in\mathbb{N}
\end{align}
for every $n\in\mathbb{N}$ and this is equivalent to
\begin{align}\label{inq9}
\left\langle f(x_n)-y_n-\alpha (x_n-x^*),f(x_n)-y_n+f(x^*)-f(x_n)+s(x_n)-s(x^*)\right\rangle\leq 0.
\end{align}

Simplifying (\ref{inq9}), we get
\begin{align}\label{inq10}
\|f(x_n)-y_n\|^2\leq & \left\langle f(x_n)-y_n, f(x_n)-f(x^*)\right\rangle+\left\langle y_n-f(x_n), s(x_n)-s(x^*)\right\rangle\\&
+\alpha\left\langle x_n-x^*, f(x_n)-y_n\right\rangle-\alpha\left\langle x_n-x^*, f(x_n)-f(x^*)\right\rangle\nonumber\\&+\alpha\left\langle x_n-x^*, s(x_n)-s(x^*)\right\rangle\nonumber.
\end{align}

Using Young's inequality, that is, $\left\langle x, y\right\rangle\leq \frac{\|x\|^2}{2}+\frac{\|y\|^2}{2}$, from (\ref{inq10}) we obtain
\begin{align}\label{inq11}
\alpha\left\langle x_n-x^*, y_n-f(x_n)\right\rangle\leq &\frac{\|f(x_n)-f(x^*)\|^2}{2}-\alpha\left\langle x_n-x^*, f(x_n)-f(x^*)\right\rangle\\&+\frac{\|s(x_n)-s(x^*)\|^2}{2}+\alpha\left\langle x_n-x^*, s(x_n)-s(x^*)\right\rangle\nonumber.
\end{align}

Since $f$ is $L$-Lipschitz continuous and $\beta$-strongly monotone, $s$ is $l$-Lipschitz continuous. Using inequality (\ref{inq11}), we get
\begin{align}\label{inq12}
\left\langle x_n-x^*, y_n-f(x_n)\right\rangle\leq \frac{1}{\alpha}(\alpha(l-\beta)+\frac{L^2+l^2}{2})\|x_n-x^*\|^2,
\end{align}
which is equivalent to
\begin{align}\label{inq13}
2\alpha\left\langle x_n-x^*, y_n-f(x_n)\right\rangle\leq(L^2+l^2+2\alpha(l-\beta))\|x_n-x^*\|^2.
\end{align}

Using (\ref{Asump:trick1}), it is not difficult to check that $2\alpha(\beta-l)-L^2-l^2>0$. Therefore, using (\ref{inq13}), we obtain
\begin{align}\label{inq14}
\left\langle x_n-x^*, f(x_n)-y_n\right\rangle\leq- \frac{(2\alpha(\beta-l)-L^2-l^2)}{2\alpha} \|x_n-x^*\|^2.
\end{align}

Since $f$ is $\beta$-strongly monotone, using (\ref{inq10}), we get
{\small \begin{align}\label{inq16}
\alpha\left\langle x_n-x^*, y_n-f(x_n)\right\rangle\leq &- \|y_n-f(x_n)\|^2+ \|y_n-f(x_n)\|\|f(x_n)-f(x^*)\|\\&+\|y_n-f(x_n)\|\|s(x_n)-s(x^*)\|\nonumber\\&
-\alpha\beta\|x_n-x^*\|^2\nonumber+\alpha\|x_n-x^*\|\|s(x_n)-s(x^*)\|\nonumber.
\end{align}}

We know that $f$ and $s$ are Lipschitz continuous with constants $L$ and $l$, respectively. Therefore, we infer from (\ref{inq16}) that
{\small \begin{align}\label{inq17}
\alpha\left\langle x_n-x^*, y_n-f(x_n)\right\rangle & \leq - \|y_n-f(x_n)\|^2\\&+ \frac{L}{\sqrt{\alpha(\beta-l)}}\|y_n-f(x_n)\|\nonumber\|x_n-x^*\|\sqrt{\alpha(\beta-l)}\\&+\frac{l}{\sqrt{\alpha(\beta-l)}}\|y_n-f(x_n)\|\|x_n-x^*\|\sqrt{\alpha(\beta-l)}\nonumber\\&
-\alpha\beta\|x_n-x^*\|^2\nonumber+\alpha l\|x_n-x^*\|^2\nonumber.
\end{align}}

Using the inequality $ab\leq\frac{a^2}{2}+\frac{b^2}{2}$ in (\ref{inq17}), we obtain
{\small \begin{align}\label{inq17*}
\alpha\left\langle x_n-x^*, y_n-f(x_n)\right\rangle & \leq - \|y_n-f(x_n)\|^2\\&+ \frac{L^2}{2\alpha(\beta-l)}\|y_n-f(x_n)\|^2\nonumber+\frac{\alpha(\beta-l)}{2}\|x_n-x^*\|^2\\&+\frac{l^2}{2\alpha(\beta-l)}\|y_n-f(x_n)\|^2+\frac{\alpha(\beta-l)}{2}\|x_n-x^*\|^2\nonumber\\&
-\alpha\beta\|x_n-x^*\|^2\nonumber+\alpha l\|x_n-x^*\|^2\nonumber.
\end{align}}

This implies that  
\begin{align}\label{inq18}
\alpha\left\langle x_n-x^*, y_n-f(x_n)\right\rangle\leq\frac{(L^2+l^2)-2\alpha(\beta-l)}{2\alpha(\beta-l)}\|y_n-f(x_n)\|^2.
\end{align}

Using assumption (\ref{Asump:trick1}) in (\ref{inq18}), we obtain
\begin{align}\label{inq19}
\|y_n-f(x_n)\|^2\leq \frac{-2\alpha^2(\beta-l)}{2\alpha(\beta-l)-(L^2+l^2)}\left\langle x_n-x^*, y_n-f(x_n)\right\rangle.
\end{align}

That is,
{\small \begin{align}\label{inq20}
\|P_{\Phi(x_n)}(f(x_n)-\alpha x_n)-f(x_n)\|^2\leq \frac{-2\alpha^2(\beta-l)}{2\alpha(\beta-l)-(L^2+l^2)}\left\langle x_n-x^*, y_n-f(x_n)\right\rangle.
\end{align}}

Substituting (\ref{inq20}) in (\ref{maininq}), we get
{\small \begin{align}\label{inq21}
\|x_{n+1}-x^*\|^2\leq &\|x_n-x^*\|^2+\Big(2\lambda_n-\frac{2\lambda_n^2\alpha^2(\beta-l)}{2\alpha(\beta-l)-(L^2+l^2)}\Big)\left\langle x_n-x^*, y_n-f(x_n)\right\rangle.
\end{align}}

Hence, using (\ref{inq14}), (\ref{Asump:trick2}) and (\ref{Asump:trick3}) in (\ref{inq21}) we obtain the following inequality:
{\tiny \begin{align}\label{inq21*}
\|x_{n+1}-x_n\|^2\leq &\|x_n-x^*\|^2-\Big(2\lambda_n-\frac{2\lambda_n^2\alpha^2(\beta-l)}{2\alpha(\beta-l)-(L^2+l^2)}\Big)\Big(\frac{2\alpha(\beta-l)-L^2-l^2}{2\alpha}\Big)\|x_n-x^*\|^2.
\end{align}}

This implies that
{\tiny
\begin{align}\label{inq22}
\|x_{n+1}-x^*\|\leq \sqrt{1-\Big(2\lambda_n-\frac{2\lambda_n^2\alpha^2(\beta-l)}{2\alpha(\beta-l)-(L^2+l^2)}\Big)\Big(\frac{2\alpha(\beta-l)-L^2-l^2}{2\alpha}\Big)}\|x_n-x^*\|.
\end{align}}

Define for every $n\in\mathbb{N}$,
\begin{align*}
 Q(\alpha, \lambda_n) := \sqrt{1-\Big(2\lambda_n-\frac{2\lambda_n^2\alpha^2(\beta-l)}{2\alpha(\beta-l)-(L^2+l^2)}\Big)\Big(\frac{2\alpha(\beta-l)-L^2-l^2}{2\alpha}\Big)}.
\end{align*}
%(\textcolor{red}{$Q(\alpha)$ can be viewed as a cubic polynomial in $\alpha$)}.

Then (\ref{inq22}) can be rewritten as
\begin{align}\label{inq23}
\|x_{n+1}-x^*\|\leq Q(\alpha, \lambda_n)\|x_n-x^*\|.
\end{align}

Finally, we get
\begin{align}\label{inq24}
\|x_{n+1}-x^*\|&\leq Q(\alpha, \lambda_n)\|x_n-x^*\|\\
&\vdots\nonumber\\
&\leq Q^n(\alpha, \lambda_n)\|x_0-x^*\|\nonumber.
\end{align}

Let $C_1=\frac{2\alpha(\beta-l)-(L^2+l^2)}{\alpha}$ and $C_2=\alpha(\beta-l)$. Then we have
\begin{align}\label{inQ}
 Q(\alpha, \lambda_n)^2=1+\lambda_n^2 C_2-\lambda_n C_1.
\end{align}

Using facts (\ref{Asump:trick1}), (\ref{Asump:trick2}) and (\ref{Asump:trick3}), we get
\begin{align}
Q(\alpha, \lambda_n)^2&
<1+B^2 C_2-A C_1=r<1\nonumber.
\end{align}

Again, from (\ref{inq24}) we infer that
\begin{align}\label{inq25}
0\leq\|x_{n+1}-x^*\|< r^{n/2}\|x_0-x^*\|\quad\forall n\in\mathbb{N}.
\end{align}

Therefore, from (\ref{inq25}) we conclude that the sequence $\{x_n\}$ is bounded and furthermore,
\begin{align*}
\|x_{n+1}-x^*\|\rightarrow 0\text{ as } n\rightarrow\infty.
\end{align*}
That is, the sequence $\{x_n\}$ generated by the algorithm \ref{mainin1} converges strongly to $x^*$, as asserted.
\end{proof}

\begin{remark}\label{He:rem}
In particular, if we take $\lambda_n=\frac{1}{\alpha}$ for every $n\in\mathbb{N}$ and $\Phi(x)=\Omega$ (that is, $s(x)=0$) for every $x\in\mathbb{R}^n$, then (\ref{Sub:mainin1}) reduces to the algorithm
\begin{align}\label{He:rem:algo}
x_{n+1}=x_n+\frac{1}{\alpha}\{P_{\Omega}(f(x_n)-\alpha x_n)-f(x_n)\},
\end{align}
which was studied by He et al. \cite{XHE2011}.

In this case, the assumptions (\ref{Asump:trick1}), (\ref{Asump:trick2}) and (\ref{Asump:trick3}) reduce to $\alpha>\frac{L^2}{\beta}$ and the sequence $\{x_n\}$ satisfies the following inequality:
\begin{align}\label{inq26}
\|x_{n+1}-x^*\|\leq\sqrt{\Big(1-\frac{\alpha\beta-L^2}{\alpha^2}\Big)}\|x_n-x^*\|,
\end{align}
where $\beta$ and $L$ are the strong monotonicity and Lipschitz constants, respectively, of the function $f$.
Therefore, from (\ref{inq26}) it is clear that the iterative sequence $\{x_n\}$ generated in (\ref{He:rem:algo}) converges to $x^*$ linearly.
\end{remark}

\begin{remark}
In this section we have considered a more general gradient projection method. We have proved that the sequence generated by algorithm \ref{Sub:mainin1} converges to the unique solution of the IQVIP (\ref{IQVIP:P}) under the strong monotonicity and Lipschitz continuity assumptions on the single-valued mapping $f$. Furthermore, in Remark \ref{He:rem} we have shown that our gradient projection method reduces to the method of \cite{XHE2011} under certain constraint qualifications and that the sequence $\{x_n\}$ generated in (\ref{He:rem:algo}) satisfies the error bound in (\ref{inq26}). Consequently, if  $\alpha>\frac{L^2}{\beta}$, then the sequence $\{x_n\}$ generated in (\ref{He:rem:algo}) converges to the unique solution $x^*$ linearly.
\end{remark}

%================================================================================================================
\section{An example}
\label{Sec:7}

\begin{example}
Let $\Omega=B[0, 1]\subset\mathbb{R}^3$, the closed unit ball centered at the origin. Consider the functions $f(x)=2x$ and $s(x)=x/4$ from $\mathbb{R}^3$ into itself. Then it is not difficult to check that $f$ is $L$-Lipschitz continuous and $\beta$-strongly monotone, where $L=2$ and $\beta=2$. Let $\alpha=2$ and $\lambda(t)=1+t^3$, where $t\geq 0$. Then we have $\lambda(t)\in C([0,+\infty))$ and
\begin{align}
\int_{t_0=0}^\infty \lambda(t)dt=+\infty\nonumber.
\end{align}

Let $\Phi(x)=s(x)+\Omega$. Since $s(x)$ is $1/4$-Lipschitz continuous, we have
\begin{align*}
\|P_{\Phi(x)}(z)-P_{\Phi(y)}(z)\| \leq\frac{1}{4}\|x-y\|\quad\forall x, y, z\in\mathbb{R}^n.
\end{align*}
Therefore $\kappa=1/4$.

It can be verified that the above parameters satisfy the following conditions:
\begin{align}
&1+2\kappa-2\beta+\alpha^2+L^2-2\alpha\beta<0\quad\text{and}\nonumber\\&
 L^2-2\alpha(\beta-\kappa)<\kappa^2\nonumber.
\end{align}

Using Theorem \ref{Thm:IQVIP}, it not difficult to check that $(0, 0, 0)$ is the unique equilibrium point for the neural network (\ref{IQVIP}), that is, $(0, 0, 0)$ is the unique solution of the IQVIP (\ref{IQVIP:P}). According to Theorem \ref{Result4}, the neural network is globally asymptotically and exponentially stable at $(0, 0, 0)$. The graph below shows that the trajectories of (\ref{IQVIP}) globally converge to the optimal solution $(0, 0, 0)$ with different starting points. Furthermore, we see that the corresponding neural network converges at a faster rate.
\begin{center}
\includegraphics[width=11cm,height=7.5cm]{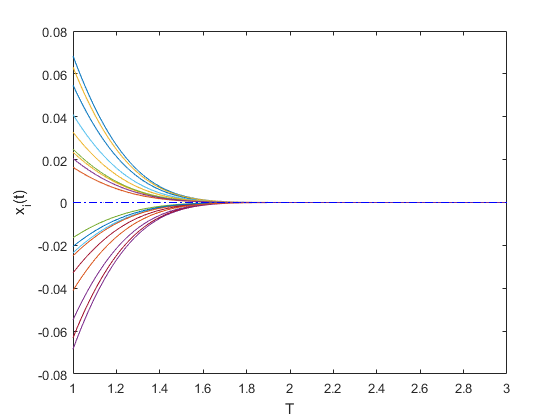}
\end{center}
\hspace{2.3cm}
{\small {\bf Fig. 1} Transient behavior of the neural network \eqref{IQVIP}}.\\
\end{example}

%=================================================================================================================
\section{conclusion}
\label{Sec:8}
This paper presents a recurrent neural network model for solving inverse quasi-variational
inequality problems. This model possesses a simple single-layer structure and low complexity
for implementation. Using the Lyapunov theory functional differential equations, we have established, under certain conditions,
the existence of the solution to the proposed network, as well as its asymptotic stability exponential stability. Also, we have proved that the sequence generated by the discretization of the network (\ref{IQVIP}) converges to the solution of the IQVIP (\ref{IQVIP:P}) under certain assumptions on the parameters involved. Finally, we have provided a numerical example to illustrate our theoretical analysis.

%================================================================================================================
\begin{remark}
Our rather technical proof of Theorem \ref{Dis:Thm} leads us to ask the following question:

Does Theorem \ref{Dis:Thm} hold for more general set-valued mappings?

This question remains open.
\end{remark}

%==================================================================================================================

\noindent
 {\bf Acknowledgments:} The first author gratefully acknowledges the financial support of the Post-Doctoral Program at the Technion - Israel Institute of Technology. 
The second author was partially supported by the Israel Science Foundation (Grant No. 820/17),
by the Fund for the Promotion of Research at the Technion and by the Technion General Research Fund.


\begin{thebibliography}{99}
\bibitem{BHE2010} B. He, X. He, H. X. Liu, Solving a class of constrained `black-box' inverse variational inequalities, Eur. J. Oper. Res., 204(3) (2010) 391-401.

\bibitem{DAUS2013}
D. Aussel,  R. Gupta, A. Mehra, Gap functions and error bounds for inverse quasi-variational inequality problem, J. Math. Anal. Appl., 407 (2013) 270-280.

\bibitem{YHAN2017} Y. Han, N. Huang, J. Lu, Y. Xiao, Existence and stability of solutions to inverse variational inequality problems, Appl. Math. Mech., 38(5) (2017) 749-764.

\bibitem{XHE2011} X. He, H. X. Liu, Inverse variational inequalities with projection-based solution methods, Eur. J. Oper. Res., 208 (2011) 12-18.

\bibitem{DKIN1980} D. Kinderlehrer, G. Stampacchia, An Introduction to Variational Inequalities and Their Applications, SIAM Academic Press, New York, 1980.

\bibitem{GSTA1964} G. Stampacchia, Formes bilineaires coercitives sur les ensembles convexes, C. R. Acad. Sci. Paris, 258 (1964) 4413-4416.

\bibitem{YCEN2010} Y. Censor, A. Gibali, S. Reich, The split variational inequality problem, The Technion-Israel Institute of Technology, Haifa $(2010)$, arXiv:1009.3780.

\bibitem{QLDO2017} Q. L. Dong, Y. Y. Lu, J. Yang, S. He, Approximately solving multi-valued variational inequalities by using a projection and contraction algorithm,  Numer. Algorithms, 76(3) (2017) 799-812.

\bibitem{FFAC2003} F. Facchinei, J. S. Pang, Finite-Dimensional Variational Inequalities and Complementarity Problems, Volume I. Springer, New York, 2003.
%++
\bibitem{MDNO2003} M. D. Noor, Well-posed variational inequalities, J. Appl. Math. $\&$ Computing, 11(1-2) (2003) 165-172.

\bibitem{YCEN2012} Y. Censor, A. Gibali, S. Reich, Algorithms for the Split Variational Inequality Problem, Numer. Algorithms, 59 (2012) 301-323.

%\bibitem{QLDO2017} Q. L. Dong, Y. Y. Lu, J. Yang, S. He, Approximately solving multi-valued variational inequalities by using a projection and contraction algorithm, Numer. Algorithms, 76(3) (2017) 799-812.

\bibitem{HYAM2001} H. Yamada, The hybrid steepest descent method for the variational inequality problem over the intersection of fixed point sets of nonexpansive mappings, In: Inherently parallel algorithms in feasibility and optimization and their applications, 8(1) (2001) 473-504.

\bibitem{YCEN2011} Y. Censor, A. Gibali, S. Reich, Strong convergence of subgradient extragradient methods for the variational inequality problem in Hilbert space, Optim. Methods Softw., 26 (2011) 827-845.

\bibitem{YCEN22011}  Y. Censor, A. Gibali, S. Reich, The subgradient extragradient method for solving variational inequalities in Hilbert space, J. Optim. Theory Appl., 148 (2011) 318-335.

\bibitem{BSHE1997} B. S. He, A class of projection and contraction methods for monotone variational inequalities, Appl. Math. Optim., 35 (1997) 69-76.

\bibitem{QLDO2019} Q. L. Dong, J. F. Yang, H. B. Yuan, The projection and contraction algorithm for solving variational inequality problems in Hilbert space, J. Nonlinear Convex Anal., 20(1) (2019) 111-122.
%++

\bibitem{SDEY2019} S. Dey, V. Vetrivel, H. K. Xu, A neural network method for monotone variational inclusions, J. Nonlinear Convex Anal., 20(11) (2019) 2387-2395.

\bibitem{NMIJ2019} N. Mijajlovi, M. Jacimovi, M. A. Noor, Gradient-type projection methods for quasi-variational inequalities, Optm. Lett., 13 (2019) 1885-1896.

\bibitem{YNES2011} Y. Nesterov, L. Scrimali, Solving strongly monotone variational and quasi-variational inequalities, Discrete Contin. Dyn. Syst., 31(4) (2011) 1383-1396.

\bibitem{XZOU2016} X. Zou, D. Gong, L. Wang, Z. Chen, A novel method to solve inverse variational inequality problems based on neural networks, Neurocomputing, 173 (2016) 1163-1168.

\bibitem{XHU2012} X. Hu, J. Wang, solving the assignment problem using continuous-time and discrete-time improved dual network, IEEE Trans. Neural Netw. Learn. Syst., 23 (2012) 821-827.

\bibitem{RHU2016} R. Hu, Y. -P. Fang, Levitin-Polyak well-posedness by perturbations for the split inverse variational inequality problem, J. Fixed Point Theory Appl., 18(4) (2016) 785-800.

\bibitem{SDEY2018} S. Dey, V. Vetrivel, On approximate solution to the inverse quasi-variational inequality problem, Sci. Math. Jpn.,  81(3) (2018) 301-306.

\bibitem{SSCH2021} S. S. Chang, Salahuddin, M. Liu, X. R. Wang, J. F. Tang, Error bounds for generalized vector inverse quasi-variational inequality problems with point to set mappings, AIMS Math., 6 (2) (2021) 1800-1815.

\bibitem{ZBWA2019} Z. B. Wang, Z. Y. Chen, Z. Chen, Gap functions and error bounds for vector inverse mixed quasi-variational inequality problems, Fixed Point Theory Appl., 2019, Paper No. 14, 14 pp.

\bibitem{HGZH2010} H. G.  Zhang, Z.W. Liu, G. B. Huang, Z. Wang, Novel weighting-delay-based stability criteria for recurrent neural networks with time-varying delay, IEEE Trans. Neural Netw., 21 (2010) 91-106.

\bibitem{QSLI2013} Q. S. Liu, C.Y. Dang, T.W. Huang, A one-layer recurrent neural network for real-time portfolio optimization with probability criterion, IEEE Trans. Cybern., 43 (2013) 14-23.

\bibitem{HZHA2011} H. Zhang, J. Liu, D. Ma, Z. Wang, Data-core-based fuzzy min-max neural net-work for pattern classification, IEEE Trans. Neural Netw.,  22 (2011) 2339-2352.

\bibitem{XBGA2005}
X. B. Gao, L. Z. Liao, L. Qi, A novel neural network for variational inequalities with linear and nonlinear constraints, IEEE Transactions on Neural Networks, 16(6) (2005) 1305-1317.

\bibitem{LVNG2019} L. V. Nguyen, X. Qin, Some results on strongly pseudomonotone quasi-variational inequalities, Set-Valued Var. Anal., 28 (2020) 239-257.

\bibitem{HKXU2020} H. K. Xu, S. Dey, V. Vetrivel, Notes on a neural network approach to inverse variational inequalities, Optimization, 70(5-6) (2021) 901-190.

\bibitem{HHBA2011} H. H. Bauschke, P. L. Combettes, Convex analysis and monotone operator theory in Hilbert spaces, Springer, New York, 2011.

\bibitem{PBEE1975} P. Beesack, Gronwall Inequalities, Carleton Mathematical Lecture Notes, No. 11, Carleton University, Ottawa, 1975.

\bibitem{MANR1985} M. A. Noor, An iterative scheme for a class of quasi-variational inequalities, J. Math.  Anal. Appl., 110 (1985) 463-468.

\bibitem{HART2002}  P. Hartman, Ordinary Differential Equations, Classics in Applied Mathematics, Vol. 18. SIAM, Philadelphia, 2002.
\end{thebibliography}
\end{document}